\documentclass[a4paper,12pt]{amsart} 
\usepackage[centertags]{amsmath}
\usepackage[margin=3cm]{geometry}
\usepackage{amsfonts}
\usepackage{amsbsy}
\usepackage{amssymb}
\usepackage{enumitem}
\usepackage{t1enc}
\usepackage[utf8]{inputenc}
\usepackage[mathscr]{eucal}
\usepackage{hyperref}

\makeatletter
\@namedef{subjclassname@2020}{%
  \textup{2020} Mathematics Subject Classification}
\makeatother

\newtheorem{theorem}{Theorem}
\newtheorem{lem}{Lemma}
\newtheorem{proposition}{Proposition}
\newtheorem{corollary}{Corollary}

\newtheorem{remark}{Remark}

\newtheorem{question}{Question}
\newtheorem{problem}{Problem}

\def \N{\mathbb N}
\def \Zn{\mathbb{Z^-}}

\def \Rp{\mathbb R^+}
\def \P{\mathbb P}
\def\T{{\mathrm{T}}}

\let\phi\varphi

\newcommand{\Diag}{{\rm Diag}\mathbb{P}\hspace{0.05cm}}

\begin{document}
	
\title    [On distribution of subsequences of primes having prime indices]
          {On distribution of subsequences of primes having prime indices with respect to the $(R)$-denseness and convergence exponent}

\author[P. Miska]{Piotr Miska}


%

%
\author[J. T. T\'oth]{J\'anos T. T\'oth}

%
%

%
\author[B. Żmija]{Błażej Żmija}
%
%

\begin{abstract} Denote by $\N$ and $\P$  the set of all positive integers and prime numbers, respectively. Let $\P=\{p_1<p_2<\dots <p_n<\dots\}$, where $p_n$ is the $n$-th prime number. For $k\in\N$ we recursively define subsequences $(p^{(k)}_n)_{n=1}^{+\infty}$ of the sequence $(p_n)_{n=1}^{+\infty}$ in the following way:
let $p_n^{(1)}=p_n$ and $p_n^{(k+1)}=p_{p_n^{(k)}}$. In this paper we study and describe some interesting properties of the sets $\P_k=\{p_1^{(k)}<p_2^{(k)}<\dots<p_n^{(k)}<\dots\}$, $\P_n^\T=\{p_n^{(1)}<p_n^{(2)}<\dots<p_n^{(k)}<\dots\}$ and $\text{Diag}\P=\{p^{(1)}_1<p^{(2)}_2<\dots <p^{(k)}_k<\dots\}$ and their elements, for $k,n\in\N$. Especially, we check whether these sets have dense sets of ratios in $\Rp$. Moreover, we compute their exponents of convergence and asymptotics of their counting functions.
\end{abstract}

\thanks{The first author is supported by the grant of the Polish National Science Centre No. UMO-2018/29/N/ST1/00470 and the scholarship START 2019 of the Foundation for Polish Science. The second author is supported by The Slovak Research and Development Agency under the grant VEGA No. 1/0776/21. The third author is supported by Czech Science Foundation, grant 21-00420M. During the preparation of the work, the third author was a scholarship holder of the Kartezjusz program funded by the Polish National Centre for Research and Development, grant No. POWR.03.02.00-00-I001/16-00.}

\parindent=0pt
\maketitle

\section{Introduction}

It is a famous result that the set of quotients of prime numbers is dense in the set of positive real numbers. It is a motivation to wide study of denseness properties of subsets of positive integers on real half-line, see e.g. \cite{BT, S1, S2, TZ}. One can meet it as an exercise on course of number theory, see \cite[Problem 218]{DKM}, \cite[Ex. 4.19]{FR}, \cite[Ex. 7, p. 107]{P}, \cite[Thm. 4]{Rib} and also in several articles, e.g. \cite[Cor. 4]{GS-JPS}, \cite[Thm. 4]{HS}, \cite[Cor. 2]{STA} (according to the last reference, the result was known to Sierpiński, who credits it to Schinzel \cite{N}). The authors of \cite{GS-JPS} generalized this result to the subsets of prime numbers in given arithmetic progressions.

Motivated by the article \cite{HR} on ``light'' subsets of positive integers (i.e. subsets with slowly growing counting functions) we focus on the family of subsets $\P_k=\{p^{(k)}_1<p^{(k)}_2<p^{(k)}_3<...\}$, $k\in\N$, of prime numbers such that every next set contains these elements of the preceding one indexed by prime numbers. As a consequence, every next set is a zero asymptotic density subset of the preceding one. Although the sets $\P_k$ are ``lighter and lighter'' as $k$ increases, we will show that all of them have dense quotient sets in the set of positive real numbers and have convergence exponent equal to $1$. Let us notice that the set $\P_2$ was already studied in \cite{BKOeS, BB}. The authors of the mentioned papers obtained results on estimations of elements and the counting function of $\P_2$, extreme values of gaps between consecutive elements of $\P_2$, appearance of these elements in arithmetic progressions and the sum of their reciprocals. 

We will also study the sets $\P^T_n=\{p^{(k)}_n: k\in\N\}$, $n\in\N$, and $\text{Diag}\P=\{p^{(k)}_k:k\in\N\}$. We shall prove that, in the opposition to the sets $\P_k$, their quotient sets are not dense in $\Rp$. Additionally, we will show that all the sets $\P_n^T$, $n\in\N$, and $\text{Diag}\P$ have convergence exponent equal to $0$.

Finally, we will give bounds for consecutive elements of the sets $\P_n^T$, $n\in\N$, and $\text{Diag}\P$. Basing on these bounds, we shall prove that the counting functions of the sets $\P_n^T$, $n\in\N$, and $\text{Diag}\P$ are asymptotically equal.

\section{Definitions and notations}

We introduce the basic definitions and conventions that will be used throughout the paper.

Denote by $\N$, $\N_0$, $\Zn$, $\P$ and $\Rp$ the set of all positive integers, non-negative integers, non-positive integers, prime numbers and positive real numbers, respectively. For given $x\ge 1$, define the counting function of $A\subset \N$ as
$A(x)=\#\{a\leq x: a\in A \}$, and for $B\subset\Rp$ we denote $B^d$ the set of all accumulation points of the set $B$ (with respect to the natural topology on $\Rp$).
\\Denote by $R(A)=\{\frac ab: a, b\in A\}$ the ratio set or quotient set of a given subset $A$ of $\N$. We say that the set $A$ is $(R)$-dense if $R(A)$ is (topologically) dense in the set $\Rp$, i.e. $R^d (A)=\Rp$. If $A$ is not $(R)$-dense, then we will say that this is a $Q$-sparse set. Let us note that the concept of $(R)$-density was defined and first studied in the papers \cite{S1} and \cite{S2}.
\\In the following instead of $\lim_{n\to +\infty}\frac{a_n}{b_n}=1$ we will write $a_n\sim b_n$ or $a_n\sim b_n$ as $n\to +\infty$, for the sequences $(a_n)$, $(b_n)$ of positive real numbers. We also use the ''small oh'', ''big Oh'', ``small omega'' and ``theta'' notations in their standard meaning. We will use the following definition of the ``big Omega'' notation: $f(x)=\Omega(g(x))$ if and only if $g(x)=O(f(x))$. If some property holds for all values greater than some constant we write that this property holds for $x\gg 0$.

The set $\P=\{p_1<p_2<\dots <p_n<\dots\}$, where $p_n$ is the $n$-th prime number fulfills the above well known properties (see \cite{HW} and \cite{NAT}):
\begin{equation}\label{e1}
p_n\sim n\log n, \quad
 p_{n+1}\sim p_n ,\quad \log p_n\sim\log n , 
\end{equation}
and
\begin{equation}\label{e2}
\P(x)=\pi(x)\sim\frac{x}{\log x} .
\end{equation}

Let us denote $p^{(0)}_n=n$ and define recursively $p^{(k)}_n$ by
$$p_n^{(k+1)}=p_{p_n^{(k)}}\,\quad\text{for }k\in\N_0.$$
Note that for every $k\in\N_0$ we have
$$p^{(k+1)}_n=p_{p^{(k)}_n}=p^{(k)}_{p_n}.$$
Further denote
$$\P_k=\{p_1^{(k)}<p_2^{(k)}<\dots<p_n^{(k)}<\dots\}\quad\text{for }\ k\in\N_0,\ \textrm{and } $$
$$\P_n^\T=\{p_n^{(1)}<p_n^{(2)}<\dots<p_n^{(k)}<\dots\}\quad\text{for }\ n\in\N\,.$$
Therefore, $\P_0=\N$, $\P_1=\P$ and for $k\in\N_0$ we have
$$\P_{k+1}=\{p_n:\   n\in \P_k\}=\{p^{(k)}_n:\  n\in \P_1\},$$
and obviously $$\P_{k+1}\subsetneq \P_k.$$

\section{Results}

\subsection{Preliminaries}

We will use the following properties to prove the results of the paper.

\begin{proposition}\label{P1}
	Let $(a_n)$, $(b_n)$, $(c_n)$ and $(d_n)$ be sequences of positive real numbers such that $a_n\sim b_n$ and $c_n\sim d_n$. Then:
	\begin{enumerate}
		\item[i)]  $a_n c_n\sim b_n d_n ,$
		
		\item[ii)] if $(a_n)$ is unbounded, then $\log a_n\sim\log b_n ,$
		
		\item[iii)] $\sum_{n=1}^{+\infty}a_n<+\infty\iff\sum_{n=1}^{+\infty}b_n<+\infty .$
	\end{enumerate}
\end{proposition}

\begin{proposition}\label{P2}
	Let $A=\{a_1<a_2<\cdots<a_n<\dots\}\subset\N$. Then:
	\begin{enumerate}
		\item[i)]  if $\lim_{n\to +\infty}\frac{a_{n+1}}{a_n}=1$ then the set $A$ is $(R$)-dense (see \cite{STA}, \cite{HS}),
		\item[ii)]  if $\liminf_{n\to +\infty}\frac{a_{n+1}}{a_n}=c>1$ then $R^d (A)\cap(\frac1c, c)=\emptyset$
		(see \cite{TZ}, Th3.).
	\end{enumerate}
\end{proposition}

\begin{proposition}\label{P3}
	 We have:
	\begin{enumerate}
		\item[i)] $p_n\geq n\log n$ for $n\in\N$  (see \cite{RO1}),
		\item[ii)] $p_n\leq n(\log n +\log\log n)$ for $n\geq6$ (see \cite{RO2}).
	\end{enumerate}
\end{proposition}

\subsection{The sets $\P_k$, $k\in\N_0$}

We start with the result which generalizes \eqref{e1}.

\begin{theorem}\label{T1}
	For every $k\in\N_0$ we have:
	\begin{enumerate}
		\item[i)]  $p^{(k)}_n\sim n\log^k n\,,$
		
		\item[ii)] $p^{(k)}_{n+1}\sim p^{(k)}_n\,,$
	
		\item[iii)] $\log p^{(k)}_n\sim\log n\,.$
	\end{enumerate}
\end{theorem}

\begin{proof}{i)} We prove the first part of the statement of the theorem by induction on $k\in\N_0$. For $k=0$ the result is obvious. In the inductive step, assuming validity i) for some~$k\geq1$ and using \eqref{e1} combined with Proposition~\ref{P1}.{i)} {ii)}, we obtain
\begin{align*}
& p_n^{(k+1)}=p_{p_n^{(k)}}\sim p_n^{(k)}\log p_n^{(k)}\sim n\log^k n\log (n\log^k n)\\
& =n\log^k n(\log n +k\log\log  n)) \sim n\log^k n\log n=n\log^{k+1} n\,.
\end{align*}
{ii)} 	We have
$$\frac{p^{(k)}_{n+1}}{p^{(k)}_n}=\frac{p^{(k)}_{n+1}}{(n+1)\log^k {(n+1)}}\ \frac{n\log^k n}{p^{(k)}_n}\ \frac{n+1}{n}\ \left(\frac{\log(n+1)}{\log n}\right)^k .$$
Then, from this equality and from {i)} we obtain {ii)}. 
\\{iii)} A direct calculation gives us the last part of the statement of our theorem: $$\log p_n^{(k)}\sim \log (n\log^k n)=\log n +k\log\log  n \sim\log n.$$
\end{proof}

In the context of Theorem \ref{T1}, it is interesting to ask the following question.

\begin{question}\label{Q1}
Is it true that $p^{(k)}_{k+1}\sim p^{(k)}_k$ as $k\to +\infty$?
\end{question}

At this moment we are ready to show that the sets $\P_k$, $k\in\N_0$, are (R)-dense. Moreover, we prove the asymptotics of elements of these sets and their counting functions.

\begin{corollary}\label{C1}
	For every $k\in\N$ we have:
	\begin{enumerate}
		\item[i)] the set $\P_k$ is (R)-dense,
		\item[ii)] $\log p_{n+1}^{(k)}\sim\log p_n^{(k)}$ as $n\to +\infty$,
		\item[iii)] $p_n^{(k+1)}\sim p_n^{(k)}\log p^{(k)}_n\sim  p_n^{(k)}\log n.$
	\end{enumerate}
\end{corollary}

\begin{proof}
Part {i)}.	This is a direct corollary of Theorem~\ref{T1}.{ii)} and Proposition~\ref{P2}.{i)}.
\\Part {ii)}. Follows from Theorem~\ref{T1}.{ii)}.
\\Part {iii)}. Follows from \eqref{e1} and Theorem~\ref{T1}.{iii)}.
\end{proof}

\begin{theorem}\label{T2}
	For every $k\in\N_0$ we have
	$$\P_k(x)\sim\frac{x}{\log^k x}\quad\textrm{as }\quad x\to +\infty .$$
\end{theorem}

\begin{proof}
	Let $k\in\N_0$ be fixed and $x\geq p^{(k)}_1$. Then, there exists an $n\in\N$ such that \\$p^{(k)}_n\leq x<p^{(k)}_{n+1}$. Thus $\P_k (x)=n$. Denote by $$H_k (x)=\frac{\P_k (x)\log^k x}{x}=\frac{n\log^k x}{x} ,$$
	and
	$$\frac{p^{(k)}_n}{p^{(k)}_{n+1}}\frac{n\log^k n}{p^{(k)}_n}\leq \frac{n\log^k x}{x}\leq\frac{n\log^k n}{p^{(k)}_n}\Bigg(\frac{\log p^{(k)}_n}{\log n}\Bigg)^k \Bigg(\frac{\log p^{(k)}_{n+1}}{\log p^{(k)}_n}\Bigg)^k .$$ 
	Then, from Theorem~\ref{T1}.{ii)} and Theorem~\ref{T1}.{i)} we deduce that the lower bound of $H_k (x)$ tends to $1$ as $x\to +\infty$. From  Theorem~\ref{T1}.{i)}, Theorem~\ref{T1}.{iii)} and Corollary~\ref{C1}.{ii)} we conclude that the upper bound of $H_k (x)$ tends to $1$ as $x\to +\infty$. Therefore $H_k (x)\to 1$ as $x\to +\infty$.
\end{proof}

Let us introduce the convergence exponent of any set $A\subset\N$:
$$\rho(A)=\inf\left\{\alpha\in [0,+\infty): \sum_{n\in A}n^{-\alpha}<+\infty\right\}.$$
Let us notice that $\rho(A)\in [0,1]$ as $\rho(\N)=1$ and $\rho(A)\leq\rho(B)$ for any $A\subset B\subset\N$. Moreover, if $\alpha>\rho(A)$, then $\sum_{n\in A}n^{-\alpha}<+\infty$ and if $\alpha<\rho(A)$, then $\sum_{n\in A}n^{-\alpha}=+\infty$. One can find in \cite[p. 41]{PS} useful formula
$$\rho(A)=\limsup_{n\to +\infty}\frac{\log n}{\log a_n},$$
where $A=\{a_1<a_2<a_3<\ldots\}$. Theorem \ref{T1}.iii) allows us to show that the sets $\P_k$, $k\in\N_0$ are big with respect to the function $\rho$, i.e. $\rho(\P_k)=1$ for each $k\in\N_0$. Indeed, we have
$$\rho(\P_k)=\limsup_{n\to +\infty}\frac{\log n}{\log p^{(k)}_n}=1.$$
\\Furthermore, we know that $\sum_{p\in \P_k}\frac1p=+\infty$ in the case of $k\in\{0,1\}$. On the other hand, $\sum_{p\in \P_k}\frac1p$ is convergent for $k\geq 2$. That is why the following theorem is interesting.

\begin{theorem}\label{T3}
	The series $$S_k^{(\alpha)}=\sum_{p\in \P_k}\frac{1}{p^{\alpha}}$$
	is convergent if and only if $\alpha>1$ or $\alpha=1$ and $k\geq 2$. Moreover, for each $\alpha\geq 1$ the value $S_k^{(\alpha)}$ tends to $0$ as $k\to +\infty$.
\end{theorem}

\begin{proof}
Since $\P_{k+1}\subsetneq\P_k$ then obviously $S^{(\alpha)}_{k}>S^{(\alpha)}_{k+1}$ if $S^{(\alpha)}_{k+1}<+\infty$. Using Theorem~\ref{T1}.{i)} and  Proposition~\ref{P1}.{iii)} we easily check for which tuples $(k,\alpha)\in\N_0\times [0,+\infty)$ the series $S^{(\alpha)}_k$ is convergent. Indeed, by integral or condensation criterion we test the convergence of the series $$\sum_{n=2}^{+\infty}\frac{1}{\left(n\log^k n\right)^{\alpha}}.$$
\\Let $k\geq2$ and $\alpha\geq 1$ be fixed. Then, from Proposition~\ref{P3}.{i)}, for every $n\in\N$ we have
$$p^{(k+1)}_n=p_{p^{(k)}_n}\geq p^{(k)}_n\log p^{(k)}_n ,$$ 
so $$\sum_{n=1}^{+\infty}\frac{1}{\left(p^{(k+1)}_n\right)^{\alpha}}\leq \sum_{n=1}^{+\infty}\frac{1}{\left(p^{(k)}_n\log p^{(k)}_n\right)^{\alpha}}<\frac{1}{\left(\log p^{(k)}_1\right)^{\alpha}}\sum_{n=1}^{+\infty}\frac{1}{\left(p^{(k)}_n\right)^{\alpha}} .$$
Since the number 
$$\frac{1}{\log p^{(k)}_1}\leq c=\frac{1}{\log p^{(2)}_1}=\frac{1}{\log 3}<1 ,$$
then for every $k\geq2$ we have
$$S^{(\alpha)}_{k+1}<c^{\alpha}S^{(\alpha)}_k.$$ Thus,
$$0\leq\lim_{k\to +\infty}S^{(\alpha)}_k\leq \lim_{k\to +\infty}c^{\alpha(k-2)}S^{(\alpha)}_2=0.$$
\end{proof}

Theorem \ref{T3} allows us to prove the following.

\begin{corollary}\label{C2}
The set $\bigcap_{k=0}^{+\infty}\P_k$ is empty.	
\end{corollary}

\begin{proof}
Assume by the contrary that
$$\bigcap_{k=0}^{+\infty}\P_k\neq\emptyset .$$	
Then, there exists a $q\in\bigcap_{k=0}^{+\infty}\P_k$, hence $q\in\P_k$ for all $k\in\N_0$. Therefore  
$$0<\frac1q<\sum_{p\in \P_k}\frac1p=S^{(1)}_k\quad\textrm{for every $k$} .$$
This contradicts with Theorem~\ref{T3}.
\end{proof}

\begin{remark}
{\rm Summing up Corollary \ref{C1}, Theorem \ref{T2}, Theorem \ref{T3} and Corollary \ref{C2}, we see that, on one hand, for each $k\in\N_0$ the set $\P_k$ is a big subset of $\N$ in the sense of (R)-denseness and convergence exponent and, on the other hand, the family of sets $\{\P_k\}_{k=0}^{+\infty}$ has empty intersection and every next member of this family has $0$ asymptotic density with respect to the preceding one, i.e.
$$\lim_{x\to +\infty}\frac{\P_{k+1}(x)}{\P_k(x)}=0.$$}
\end{remark}

\subsection{The sets $\P_n^T$, $n\in\N$}

The next theorems concern the sets $\P^T_n$ and their elements.

The first result shows that the sequence $(\P^T_n)_{n\in\N}$, unlike $(\P_k)_{k\in\N}$, is not decreasing with respect to the relation of inclusion.

\begin{theorem}\label{T5}
	Let $n,m\in\N$ with $n<m$. Then,
	$$\P_n^\T\cap \P_m^\T\ne\emptyset\iff m\in \P_n^\T\,.$$
	Moreover, if $m\in \P_n^\T$, then $\P_m^\T\subset \P_n^\T$ and
	$$\P_n^\T\setminus \P_m^\T = \{p_n^{(1)},p_n^{(2)},\dots,p_n^{(k)}\}\quad\textrm{where }\quad m=p_n^{(k)}\ .$$
\end{theorem}

\begin{proof}
	Assume that $\P_n^\T\cap \P_m^\T\neq\emptyset$. Then, there exist $j_1, j_2\in\N$ such that $p^{(j_1)}_n=p^{(j_2)}_m$. Because $n<m$ we have $j_1>j_2$. Indeed, if we assume by contrary that $j_1\leq j_2$, then we will have $p^{(j_1)}_n\leq p^{(j_2)}_n<p^{(j_2)}_m$, which is a contradiction.\\ Furthermore, we can show inductively that 
	$$p^{(j_1-i)}_n=p^{(j_2-i)}_m\quad\textrm{for each }\quad i\in\{1, 2,\dots,j_2\} .$$ Indeed, writing 
	$$p^{(j_1-i)}_n=p_{p^{(j_1-i-1)}_n}\quad\textrm{and }\quad p^{(j_2-i)}_m=p_{p^{(j_2-i-1)}_m} ,$$
	and using the injectivity of numeration of elements of $\P$ we conclude that $p^{(j_1-i-1)}_n=p^{(j_2-i-1)}_m$. In particular, for $i=j_2$ we obtain $p^{(j_1-j_2)}_n=p^{(0)}_m=m$. Thus $m\in \P^T_n$. On the other hand, if we assume that $m=p^{(k)}_n\in \P^T_n$, then by simple induction on $i\in\N_0$ we show that $p^{(i)}_m=p^{(k+i)}_n$. For $i=0$ we have $m=p^{(0)}_m=p^{(k)}_n$. For $i\in\N_0$ we have 
	$$p^{(i+1)}_m=p_{p^{(i)}_m}=p_{p^{(k+i)}_n}=p^{(k+i+1)}_n .$$
	Hence, $\P^T_m=\{p^{(k+i)}_n:\  i\in\N\}$ and  the result follows.
\end{proof}

The following lemma will be useful in proving the fact that the sets $\P_n^T$, $n\in\N$, are not (R)-dense.

\begin{lem}\label{T4}
For every $n\in\N$ we have
	$$p_n^{(k+1)}\sim p_n^{(k)}\log p_n^{(k)}\ \textrm{as } k\to +\infty .$$
\end{lem}

\begin{proof}
	For a fixed $n\in\N$ we have $p^{(k)}_n\to +\infty$ as $k\to +\infty$. Then, from \eqref{e1} we obtain
	$$1=\lim_{n\to +\infty}\frac{p_n}{n\log n}=\lim_{k\to +\infty}\frac{p_{p^{(k)}_n}}{p^{(k)}_n\log p^{(k)}_n}=\lim_{k\to +\infty}\frac{p^{(k+1)}_n}{p^{(k)}_n\log p^{(k)}_n} .$$
\end{proof}

\begin{theorem}\label{C3}
	For every $n\in\N$ the set $\P_n^\T$ is not (R)-dense, hence it is $Q$-sparse. Moreover,
	$$\lim_{k\to +\infty}\frac{p_n^{(k+1)}}{p_n^{(k)}}= +\infty ,$$
	and thus each point of the set $R(\P^T_n)$ is an isolated point, i.e.
	$$R^d(\P_n^\T)\cap(0, +\infty)=\emptyset\,.$$
\end{theorem}

\begin{proof}
	This is a direct corollary of Lemma~\ref{T4} and Proposition~\ref{P2}.{ii)}.
\end{proof}

Theorem \ref{C3} states that the set $\P^T_n$ is not (R)-dense for any $n\in\N$. Let us notice that this set is also small in the sense of convergence exponent. To be more precise, $\rho(\P^T_n)=0$ for $n\in\N$.

\begin{theorem}\label{T8}
	For every $n\in\N$ and $\alpha>0$ the series $$S^{T,\alpha}_n=\sum_{p\in \P^T_n}\frac{1}{p^{\alpha}}$$
	is convergent and $S^{T,\alpha}_n\to 0$ as $n\to +\infty$. Moreover, if we put
	$$S^{T,\alpha}_n(x)=\sum_{p\in \P^T_n, p\leq x}\frac{1}{p^{\alpha}}$$
	and assume that $p^{(k-1)}_n\leq x<p^{(k)}_n$ for some integer $k\geq 2$, then
	$$S^{T,\alpha}_n-S^{T,\alpha}_n(x)\leq\frac{1}{\left(p^{(k)}_n\right)^{\alpha}}\ \frac{\left(\log p^{(k)}_n\right)^{\alpha}}{\left(\log p^{(k)}_n\right)^{\alpha} -1}.$$
\end{theorem}

\begin{proof}
Let $n\in\N$ and $x\in\left[p^{(k-1)}_n,p^{(k)}_n\right)$ be fixed. Then $0<\frac{1}{\log p^{(k)}_n}<1$. From Theorem \ref{T6.1}.i) we obtain \mbox{$p^{(j)}_n\geq p^{(k)}_n\log^{j-k} p^{(k)}_n$} and that is why
\begin{align*}
& S^{T,\alpha}_n-S^{T,\alpha}_n(x)=\sum_{j=k}^{+\infty}\frac{1}{\left(p^{(j)}_n\right)^{\alpha}}\leq\sum_{j=k}^{+\infty}\frac{1}{\left(p^{(k)}_n\right)^{\alpha}}\frac{1}{\left(\left(\log p^{(k)}_n\right)^{\alpha}\right)^{(j-k)}}\\
& =\frac{1}{\left(p^{(k)}_n\right)^{\alpha}}\ \frac{\left(\log p^{(k)}_n\right)^{\alpha}}{\left(\log p^{(k)}_n\right)^{\alpha} -1}.
\end{align*}
Putting $k=2$ we can see that $$S^{T,\alpha}_n\leq\frac{1}{\left(p^{(1)}_n\right)^{\alpha}}+\frac{1}{\left(p^{(2)}_n\right)^{\alpha}}\ \frac{\left(\log p^{(2)}_n\right)^{\alpha}}{\left(\log p^{(2)}_n\right)^{\alpha} -1}$$ is convergent and $S^{T,\alpha}_n\to 0$ if $n\to +\infty$.
\end{proof}

\subsection{The set $\text{Diag}\P$}

Let us consider the set $\text{Diag}\P=\{p^{(k)}_k:k\in\N\}$ of the elements on the diagonal of the infinite matrix
$$[p^{(k)}_n]_{n,k\in\N}=\left[\begin{array}{ccccc}
p^{(1)}_1 & p^{(2)}_1 & \dots & p^{(k)}_1 & \dots\\
p^{(1)}_2 & p^{(2)}_2 & \dots & p^{(k)}_2 & \dots\\
\vdots & \vdots & \ddots & \vdots & \ddots\\
p^{(1)}_n & p^{(2)}_n & \dots & p^{(k)}_n & \dots\\
\vdots & \vdots & \ddots & \vdots & \ddots
\end{array}\right].$$
The next two theorems state that the set $\text{Diag}\P$, as well as the sets $\P^T_n$, $n\in\N$, is not (R)-dense and its convergence exponent is $0$.

\begin{theorem}\label{T9}
	The set $\emph{Diag}\P$ is not (R)-dense, hence it is $Q$-sparse. Moreover,
	$$\lim_{k\to +\infty}\frac{p_{k+1}^{(k+1)}}{p_k^{(k)}}= +\infty ,$$
	and thus each point of the set $R\left(\emph{Diag}\P\right)$ is an isolated point, i.e.
	$$R^d\left(\emph{Diag}\P\right)\cap(0,+\infty)=\emptyset\,.$$
\end{theorem}

\begin{proof}
From Proposition \ref{P3}.{i)} we know that
$$p_{k+1}^{(k+1)}\geq p_{k+1}^{(k)}\log p_{k+1}^{(k)}>p_k^{(k)}\log p_{k+1}^{(k)}.$$
Thus,
$$\lim_{k\to +\infty}\frac{p_{k+1}^{(k+1)}}{p_k^{(k)}}\geq\lim_{k\to +\infty}\log p_{k+1}^{(k)}=+\infty$$
and the rest of the statement of the theorem follows from Proposition~\ref{P2}.{ii)}.
\end{proof}

\begin{theorem}\label{T10}
	For every $\alpha>0$ the series $$S^{(\alpha)}_{\emph{diag}}=\sum_{p\in\emph{Diag}\P}\frac{1}{\left(p^{(k)}_k\right)^{\alpha}}$$
	is convergent. Moreover, let
	$$S^{(\alpha)}_{\emph{diag}}(x)=\sum_{p\in \emph{Diag}\P, p\leq x}\frac{1}{p^{\alpha}}$$
	and assume that $p^{(k-1)}_{k-1}\leq x<p^{(k)}_k$ for some integer $k\geq 2$. Then, we have
	$$S^{\alpha}_{\emph{diag}}-S^{\alpha}_{\emph{diag}}(x)\leq\frac{1}{\left(p^{(k)}_k\right)^{\alpha}}\ \frac{\left(\log p^{(k)}_k\right)^{\alpha}}{\left(\log p^{(k)}_k\right)^{\alpha} -1}.$$
\end{theorem}

\begin{proof}
Follows directly from the convergence of the series $S^{T,\alpha}_1$ and the convergence comparison test as $p^{(k)}_k\geq p^{(k)}_1$ for each $k\in\N$.
\end{proof}

\subsection{Asymptotics of the numbers $p_{n}^{(k)}$ and $p_{k}^{(k)}$ as $k\to\infty$}

In this Section, we find explicit upper and lower bounds for $p_{n}^{(k)}$. Proposition \ref{P3} implies
\begin{align}\label{ineqPrimes}
    n\log n<p_{n}<2n\log n.
\end{align}

We start with the upper bound.

\begin{lem}\label{lemUP}
Let $n\geq 9$. Then for each $k\in\mathbb{N}$ we have:
\begin{align*}
    p_{n}^{(k)}<2^{2k-1}\cdot n\cdot (k-1)!\cdot\big(\log(\max\{k,n\})\big)^{k}. 
\end{align*}
In particular,
\begin{align*}
    p_{n}^{(k)}< \big(4\cdot k\log k\big)^{k}
\end{align*}
for $k\geq n$.
\end{lem}
\begin{proof}
We proceed by induction on $k$. For $k=1$ it is a simple consequence of (\ref{ineqPrimes}).
Then the second induction step goes as follows: let us denote $m:=\max\{n,k\}$. Observe that $(k-1)!<(k-1)^{k-1}<m^{k-1}$ and $4\log m<m$ for $m\geq 9$. Hence,
\begin{align*}
    p_{n}^{(k+1)}\leq &\ 2p_{n}^{k}\log p_{n}^{(k)} \\ 
    < &\ 2\cdot 2^{2k-1}\cdot n\cdot (k-1)!\cdot(\log m)^{k}\cdot \log\left[2^{2k-1}\cdot n\cdot (k-1)!\cdot (\log m)^{k}\right] \\
    < & 2^{2k}\cdot n\cdot (k-1)!\cdot(\log m)^{k} \log\left[4^{k}\cdot m\cdot m^{k-1}\cdot (\log m)^{k}\right] \\ 
    = &\ 2^{2k}\cdot n\cdot (k-1)!\cdot(\log m)^{k} \log\left[4\cdot m\cdot \log m\right]^{k} \\
    \leq &\ 2^{2k}\cdot n\cdot k!\cdot (\log m)^{k}\cdot \log[m]^{2}\leq 2^{2k+1}\cdot n\cdot k!\cdot (\log m)^{k+1}.
\end{align*}

The second part of the statement is an easy consequence of the first part and the inequalities $(k-1)!<k^{k-1}$ and $n\leq k$.
\end{proof}

In order to prove a lower bound for $p_{n}^{(k)}$ we will need the following fact.

\begin{lem}\label{lemL}
Let
\begin{align*}
    L(x):=\left(\frac{x}{x+1}\right)^{x+1}\left(\frac{\log x}{\log (x+1)}\right)^{x+1}.
\end{align*}
Then we have
\begin{align*}
    L(x)>0.32627
\end{align*}
for all $x\geq 4200$.
\end{lem}
\begin{proof}
Observe, that the function $\left(\frac{x}{x+1}\right)^{x+1}$ is increasing. Indeed, if
\begin{align*}
    f(x):=\log\left(\frac{x}{x+1}\right)^{x+1}=(x+1)\left(\log x-\log (x+1)\right),
\end{align*}
then
\begin{align*}
    f'(x)=\log x-\log (x+1)+(x+1)\left(\frac{1}{x}-\frac{1}{x+1}\right)=\frac{1}{x}-\log \left(1+\frac{1}{x}\right)>0,
\end{align*}
where the last inequality follows from the well-known inequality $y>\log (1+y)$ used with $y=\frac{1}{x}$. Hence, we can bound
\begin{align}\label{L1}
    \left(\frac{x}{x+1}\right)^{x+1}\geq \left(\frac{4200}{4201}\right)^{4201}
\end{align}
for all $x\geq 4200$.

Now we need to find a lower bound for $\left(\frac{\log x}{\log (x+1)}\right)^{x+1}$. Let us write
\begin{align*}
    \left(\frac{\log x}{\log (x+1)}\right)^{x+1}=\left[\left(1-\frac{1}{\frac{\log (x+1)}{\log (x+1)-\log x}}\right)^{\frac{\log (x+1)}{\log (x+1)-\log x}}\right]^{\log\left(1+\frac{1}{x}\right)^{x}\cdot\frac{1}{\log (x+1)}\cdot\left(1+\frac{1}{x}\right)}.
\end{align*}
At first, we prove that functions $g(t):=\left(1-\frac{1}{t}\right)^{t}$ and $h(x):=\frac{\log (x+1)}{\log (x+1)-\log x}$ are increasing. For the function $g(t)$ it is enough to observe, that $\log g(t)=f(t-1)$ and the function $f(x)$ is increasing. For the function $h(x)$ we have:
\begin{align*}
    h'(x)= & \frac{1}{(\log (x+1)-\log x)^{2}}\left[\frac{\log (x+1)-\log x}{x+1}-\log (x+1)\left(\frac{1}{x+1}-\frac{1}{x}\right)\right] \\
    = & \frac{1}{(\log (x+1)-\log x)^{2}}\left[\frac{\log (x+1)}{x}-\frac{\log x}{x+1}\right] >0.
\end{align*}

The fact that the functions $g(t)$ and $h(x)$ are increasing, together with the properties $g(h(4200))\in (0,1)$ and $\log\left(1+\frac{1}{x}\right)^{x}<1$, give us
\begin{equation}\label{L2}
    \begin{aligned}
        \left(\frac{\log x}{\log (x+1)}\right)^{x+1}\geq & \big[g(h(4200))\big]^{\log\left(1+\frac{1}{x}\right)^{x}\cdot\frac{1}{\log (x+1)}\cdot\left(1+\frac{1}{x}\right)} 
        > \big[g(h(4200))\big]^{\frac{1}{\log (x+1)}\cdot\left(1+\frac{1}{x}\right)} \\
        \geq & \big[g(h(4200))\big]^{\frac{1}{\log (4200+1)}\cdot\left(1+\frac{1}{4200}\right)}=\left(\frac{\log 4200}{\log 4201}\right)^{\frac{4201}{4200}\cdot\frac{1}{\log 4201 -\log 4200}}
    \end{aligned}
\end{equation}
for all $x\geq 4200$.

Combining (\ref{L1}) and (\ref{L2}) we get the inequality
\begin{align*}
    L(x)\geq \left(\frac{4200}{4201}\right)^{4201}\left(\frac{\log 4200}{\log 4201}\right)^{\frac{4201}{4200}\cdot\frac{1}{\log 4201 -\log 4200}}\approx 0.3262768>0.32627.
\end{align*}
The proof is finished.
\end{proof}

In the next lemma we provide a lower bound for $p_{n}^{(k)}$.

\begin{lem}\label{lemDOWN}
If $n>e^{4200}$, then for all $k\geq\lfloor\log n\rfloor$ we have
\begin{align*}
    p_{n}^{(k)}>\left(\frac{e\cdot k\log k}{\log\log n}\right)^{k}.
\end{align*}
\end{lem}
\begin{proof}
First, let us observe that a simple induction argument on $k$ implies the inequality
\begin{align}\label{ineqDOWN}
    p_{n}^{(k)}>n(\log n)^{k}.
\end{align}
Indeed, for $k=1$ this follows from left inequality in (\ref{ineqPrimes}). Using the same inequality we get also 
\begin{align*}
p_{n}^{(k+1)}>p_{n}^{(k)}\log p_{n}^{(k)}>n (\log n)^{k}\log(n (\log n)^{k})>n(\log n)^{k+1},
\end{align*}
and hence (\ref{ineqDOWN}).

Now we show that the inequality from the statement is true for $k=\lfloor \log n\rfloor$. Because of (\ref{ineqDOWN}) it is enough to show:
\begin{align*}
    n(\log n)^{k}>\left(\frac{e\cdot k\log k}{\log\log n}\right)^{k},
\end{align*}
or equivalently, after taking logarithms we get
\begin{align*}
    &  \log n+\lfloor\log n\rfloor\log\log n>\lfloor\log n\rfloor+\lfloor\log n\rfloor\log\lfloor\log n\rfloor \\
    & \hspace{6.5cm}  +\lfloor\log n\rfloor\log\log\lfloor\log n\rfloor -\lfloor\log n\rfloor\log\log\log n.
\end{align*}
This is equivalent to the inequality
\begin{align*}
    & \big(\log n-\lfloor\log n\rfloor\big)+\lfloor\log n\rfloor\big(\log\log n-\log\lfloor\log n\rfloor\big) \\
    & \hspace{6.5cm} +\lfloor\log n\rfloor\big(\log\log\log n-\log\log\lfloor\log n\rfloor\big)>0,
\end{align*}
which is obviously true.

In order to finish the proof, we again use the induction argument. The inequality from the statement of our lemma is true for $k=\lfloor \log n\rfloor$. Assume it holds for some $k\geq \lfloor\log n\rfloor$. Then by (\ref{ineqPrimes}) and the induction hypothesis we get
\begin{align*}
    p_{n}^{(k+1)}> & p_{n}^{(k)}\log p_{n}^{(k)}>\left(\frac{e\cdot k\log k}{\log\log n}\right)^{k}\log \left(\frac{e\cdot k\log k}{\log\log n}\right)^{k}.
\end{align*}
It is enough to show that for all $n>e^{4200}$ and all $k\geq\lfloor\log n\rfloor$ we have
\begin{align*}
    \left(\frac{e\cdot k\log k}{\log\log n}\right)^{k}\log \left(\frac{e\cdot k\log k}{\log\log n}\right)^{k}>\left(\frac{e\cdot (k+1)\log (k+1)}{\log\log n}\right)^{k+1}.
\end{align*}
This is equivalent to
\begin{align*}
    k^{k+1}(\log k)^{k}\left[\log k+\log\left(\frac{e\log k}{\log\log n}\right)\right]>\frac{e}{\log\log n}(k+1)^{k+1}(\log(k+1))^{k+1}.
\end{align*}
Recall, that we assume that $k\geq \lfloor\log n\rfloor$. Thus $k^{e}>\log n$, that is, $e\log k>\log\log n$. Therefore, it is enough to show the following inequality:
\begin{align*}
    k^{k+1}(\log k)^{k+1}>\frac{e}{\log\log n}(k+1)^{k+1}(\log(k+1))^{k+1},
\end{align*}
or equivalently
\begin{align*}
    \left(\frac{k}{k+1}\right)^{k+1}\left(\frac{\log k}{\log (k+1)}\right)^{k+1}>\frac{e}{\log\log n}.
\end{align*}
Notice that the left-hand side expression of the last inequality is equal to $L(k)$, where the function $L(x)$ is defined in the statement of Lemma \ref{lemL}. If $n>e^{4200}$, then $k\geq \lfloor\log n\rfloor \geq 4200$, and Lemma \ref{lemL} implies $L(k)>0.32627$. Therefore, if $N:=\max\left\{\lfloor e^{4200}\rfloor,\lceil e^{e^{e/0.32627}}\rceil\right\}=\lfloor e^{4200}\rfloor$, then for all $n>N$ and $k\geq \lfloor\log n\rfloor$ we have:
\begin{align*}
    L(k)>0.32627\geq \frac{e}{\log\log N}>\frac{e}{\log\log n}.
\end{align*}
This finishes the proof.
\end{proof}

We are ready to prove the main result of this section.

\begin{theorem}\label{MAIN}
\begin{enumerate}
    \item Let $n>e^{4200}$. Then
    \begin{align*}
        \log p_{n}^{(k)}=k(\log k+\log\log k+O_{n}(1))
    \end{align*}
    as $k\rightarrow\infty$, where the implied constant may depend on $n$.
    \item We have
    \begin{align*}
        \log p_{k}^{(k)}=k(\log k+\log\log k+O(\log\log\log k))
    \end{align*}
    as $k\rightarrow\infty$.
\end{enumerate}
\end{theorem}
\begin{proof}
If $n>e^{4200}$ and $k\geq n$, Lemmas \ref{lemUP} and \ref{lemDOWN} give us:
\begin{align*}
    \left(\frac{e\cdot k\log k}{\log\log n}\right)^{k}<p_{n}^{(k)}<\left(4\cdot k\log k\right)^{k}.
\end{align*}
After taking logarithms, we simply get the first part of our theorem. In order to get the second part, we need to put $n=k$ and repeat the reasoning.
\end{proof}

Theorem \ref{MAIN} implies that for every $n$ we have $\log p_{n}^{(k)}\sim \log p_{k}^{(k)}$ as $k\to\infty$. On the other hand, $p_{k}^{(k)}$ seems to grow much faster than $p_{n}^{(k)}$ for every $n$. Indeed, we show this in the next result.

\begin{proposition}\label{P4}
Let $n\in\mathbb{N}$ be fixed. Then
\begin{align*}
    \frac{p_{n}^{(k)}}{p_{k}^{(k)}}\longrightarrow 0
\end{align*}
as $k\longrightarrow\infty$.
\end{proposition}
\begin{proof}
Let $k> p_{n}$. Then
\begin{align*}
    0\leq \frac{p_{n}^{(k)}}{p_{k}^{(k)}}<\frac{p_{n}^{(k)}}{p_{p_{n}}^{(k)}}=\frac{p_{n}^{(k)}}{p_{n}^{(k+1)}}.
\end{align*}
The expression on the right goes to zero as $k$ goes to infinity, as was proved in Theorem \ref{C3}.
\end{proof}

In our opinion, it would be interesting to find the rate of growth of consecutive elements of the sets $\mathbb{P}_{n}^{T}$ and $\Diag$. Remember that Theorem \ref{MAIN} gives only asymptotics of logarithms of numbers $p_n^{(k)}$ and $p_k^{(k)}$, where $k\to\infty$.

\begin{problem}
Find the precise asymptotics for the numbers $p_{n}^{(k)}$ and $p_{k}^{(k)}$ as $k\to\infty$.
\end{problem}

\subsection{Asymptotics of counting functions of the sets $P_n^T$ and $\text{Diag}\P$}

The next results are devoted to asymptotics of counting functions of sets $\P^T_n(x)$ and their elements. Our aim is to show that for every $n$ we have 
$$\mathbb{P}_{n}^{T} (x)\sim\Diag (x)\sim\frac{\log x}{\log\log x}.$$
We begin with some preparatory results. At first, we prove that for every $m$ and $n$ we have $\mathbb{P}_{m}^{T}(x)\sim\mathbb{P}_{n}^{T}(x)$.

\begin{theorem}\label{T6}
	For any $m,n\in\N$ there exists a constant $c>0$ (depending on $m$ and $n$) such that
	$$\left|\P_m^\T(x)-\P_n^\T(x)\right|\le c\,.$$
	In particular,
	$$\P_m^\T(x)\sim \P_n^\T(x)\,$$
	for arbitrary $m,n\in\N$.
\end{theorem}
\begin{proof}
	Without loss of generality we assume that $n<m$. We already know that if $m=p^{(j)}_n$ for some $j\in\N$ then from Theorem~\ref{T6} we have $\#\left(\P_n^\T\setminus \P^T_{p^{(j)}_n}\right)=j$ and $\P^T_{p^{(j)}_n}\subset \P^T_n$. Hence 
	$$\P^T_n(x)=\P^T_{p^{(j)}_n}(x)+j\quad\textrm{for }\quad x\geq p^{(j)}_n .$$
	Now, let $n<m$ be arbitrary. Since the sequence $\big(p^{(j)}_n\big)_{j=1}^{+\infty}$ is a strictly increasing sequence of integers, there exists a $j_0\in\N_0$ such that $p^{(j_0)}_n\leq m<p^{(j_0 +1)}_n$. Thus,
	$$p^{(j_0+k)}_n\leq p^{(k)}_m<p^{(j_0+1+k)}_n\quad\textrm{for each }\quad k\in\N .$$
	As a result,
	$$\P^T_{p^{(j_0)}_n}(x)\geq \P^T_m(x)\geq \P^T_{p^{(j_0+1)}_n}(x)\quad\textrm{for }\quad x\geq0 .$$
	This implies
	$$\P^T_n(x)-j_0\geq \P^T_m(x)\geq \P^T_n(x)-j_0-1\quad\textrm{for }\quad x\geq p^{(j_0+1)}_n .$$
	In other words,
	$$\P^T_n(x)-\P^T_m(x)\in\{j_0, j_0+1\} .$$
\end{proof}

We will also need the following two technical lemmas.

\begin{lem}\label{T6.1}
For each $n\in\N$ and $k,j\in\N_0$ with $j\geq k$ the following estimations hold:
\begin{enumerate}
\item[(i)] $p^{(j)}_n\geq p^{(k)}_n\log^{j-k}p^{(k)}_n$,
\item[(ii)] $p^{(j)}_n\geq p_n^{(k)}\prod_{i=0}^{j-k-1}(\log p^{(k)}_n+i\log\log p^{(k)}_n)$.
\end{enumerate}
\end{lem}
\begin{proof}
Using  Proposition~\ref{P3}.{i)} $j-k$ times, we obtain
\begin{equation}\label{e22}
\begin{split}
& p^{(j)}_n\geq p^{(j-1)}_n\log p^{(j-1)}_n\geq  p^{(j-2)}_n\log p^{(j-2)}_n\log p^{(j-1)}_n\geq\dots\\
& \geq  p^{(k)}_n\log p^{(k)}_n\log p^{(k+1)}_n\dots\log p^{(j-1)}_n.
\end{split}
\end{equation}
Estimating $\log p^{(i)}_n$, $i\in\{k,...,j-1\}$, by $\log p^{(k)}_n$ we obtain the first inequality in the statement of the theorem.

Having the first inequality, we use it in (\ref{e22}) by estimating $\log p^{(i)}_n$, $i\in\{k,...,j-1\}$, by $\log p^{(k)}_n+i\log\log p^{(k)}_n$. Then we get the second inequality.
\end{proof}

\begin{remark}
In the case of the numbers $p_{j}^{(j)}$ one can find lower bounds similar to those from Lemma \ref{T6.1}. More precisely, for each $k,j\in\N$ with $j\geq k$ the following estimations hold:
\begin{enumerate}
\item[(i)] $p^{(j)}_j\geq p^{(k)}_k\log^{j-k}p^{(k)}_k$,
\item[(ii)] $p^{(j)}_j\geq p_k^{(k)}\prod_{i=0}^{j-k-1}(\log p^{(k)}_k+i\log\log p^{(k)}_k)$.
\end{enumerate}
\end{remark}

\begin{lem}\label{T7}
	For every $n\in\N$, $k\in\N_0$ and $x\geq p^{(k)}_n$ we have
	$$\P^T_n(x)\leq\frac{1}{\log\log p^{(k)}_n}\ \log x +\left(k-\frac{\log p^{(k)}_n}{\log\log p^{(k)}_n}\right).$$
\end{lem}
\begin{proof}
Let $n\in\N$, $k\in\N_0$ and $x\geq p^{(k)}_n$ be fixed. Then, there exists a $j\in\N$ such that \\$p^{(j)}_n\leq x<p^{(j+1)}_n$. Thus $\P^T_n (x)=j$. 
Therefore,
\begin{equation}\label{e3}
x\geq p^{(j)}_n\geq p^{(k)}_n\log^{j-k} p^{(k)}_n
\end{equation}
by Lemma \ref{T6.1}.i). Then, from \eqref{e3} we have
$$j-k\leq \frac{1}{\log\log p^{(k)}_n}\ \log x -\frac{\log p^{(k)}_n}{\log\log p^{(k)}_n} .$$
Thus,
$$\P^T_n(x)=j\leq \frac{1}{\log\log p^{(k)}_n}\ \log x +\left(k-\frac{\log p^{(k)}_n}{\log\log p^{(k)}_n}\right).$$
\end{proof}

We are ready to prove the following results concerning the asymptotics of $\mathbb{P}_{n}^{T}(x)$ and $\Diag (x)$ as $x\to\infty$. At first, we prove the asymptotic eaulity between counting functions of $\Diag$ and $P_n^T$.

\begin{theorem}\label{AsympEqThm}
For each $n\in\mathbb{N}$ we have
\begin{align*}
    \Diag (x)\sim \mathbb{P}_{n}^{T} (x)
\end{align*}
as $x\rightarrow\infty$.
\end{theorem}
\begin{proof}
From Theorem \ref{T6} we know that
\begin{align*}
    \mathbb{P}_{m}^{T}(x)\sim \mathbb{P}_{n}^{T}(x)
\end{align*}
for each $m,n\in\mathbb{N}$. Therefore, it is enough to prove $\Diag (x)\sim \mathbb{P}_{n}^{T} (x)$ for some sufficiently large $n$.

Let $n=\lfloor e^{4200}\rfloor +100$ and let $x$ be a large real number. Let $k$ be such that $p_{k}^{(k)}\leq x<p_{k+1}^{(k+1)}$. Then $\Diag (x)=k$. By Lemma \ref{T7} and Theorem \ref{MAIN} we have
\begin{align*}
    \frac{\mathbb{P}_{n}^{T}(x)}{\Diag (x)}\leq & 1+\frac{\log p_{k+1}^{(k+1)}}{k\log\log p_{n}^{(k)}}-\frac{\log p_{n}^{(k)}}{k\log\log p_{n}^{(k)}} \\
    = & 1+\frac{(1+o(1))(k+1)\log (k+1)}{k\log \big[(1+o(1))k\log k\big]}-\frac{(1+o(1))k\log k}{k\log\big[(1+o(1))k\log k\big]} \\
    = & 1+(1+o(1))\left(1+\frac{1}{k}\right)\frac{\log (k+1)}{\log k+\log\left[(1+o(1))\log k\right]} \\ 
    & \hspace{5cm} -(1+o(1))\frac{\log k}{\log k+\log\left[(1+o(1))\log k\right]}.
\end{align*}
The whole last expression goes to $1$ as $k$ goes to infinity. On the other hand, $\Diag (x)\leq \mathbb{P}_{n}^{T}(x)$ for $x\geq p_{n}^{(n)}$ and we get the result.
\end{proof}

At last, we compute an asymptotic formula for counting functions of $\Diag$ and $P_n^T$.

\begin{theorem}\label{AsympDiagP}
\begin{enumerate}
    \item Let $n\in\mathbb{N}$. Then
    \begin{align*}
        \mathbb{P}_{n}^{T}(x)\sim\frac{\log x}{\log\log x}.
    \end{align*}
    \item We have
    \begin{align*}
        \Diag (x)\sim \frac{\log x}{\log\log x}.
    \end{align*}
\end{enumerate}
\end{theorem}
\begin{proof}
In view of Theorem \ref{AsympEqThm}, it is enough to show the statement for the function $\Diag (x)$. Let us fix an arbitrarily small number $\varepsilon >0$ and take a sufficiently large real number $x$ and find $k$ such that $p_{k}^{(k)}\leq x<p_{k+1}^{(k+1)}$. Then $\Diag (x)=k$ and by Lemmas \ref{lemUP} and \ref{lemDOWN} we have
\begin{align*}
    k^{k}<x<k^{(1+\varepsilon )k}.
\end{align*}
Let us write $x=e^{y}$. Then 
\begin{align*}
    k\log k< y<(1+\varepsilon )k\log k.
\end{align*}
If $y$ is sufficiently large, this implies 
\begin{align}\label{ineqyk}
    (1-\varepsilon )\frac{y}{\log y}< k<(1+\varepsilon )\frac{y}{\log y}.
\end{align}
Indeed, if $k\leq (1-\varepsilon )\frac{y}{\log y}$, then
\begin{align*}
    y<(1+\varepsilon)k\log k\leq (1-\varepsilon^{2})\frac{y}{\log y}\log\left((1-\varepsilon )\frac{y}{\log y}\right)<(1-\varepsilon^{2})\left(1-\frac{\log\log y}{\log y}\right)y,
\end{align*}
which is impossible. Similarly, if $k\geq (1+\varepsilon )\frac{y}{\log y}$, then
\begin{align*}
    y>k\log k\geq (1+\varepsilon )\frac{y}{\log y}\log\left((1+\varepsilon )\frac{y}{\log y}\right)>(1+\varepsilon )\left(1-\frac{\log\log y}{\log y}\right)y.
\end{align*}
The above inequality cannot hold if $y$ is sufficiently large.

If we go back to $k=\Diag (x)$ and $y=\log x$ in (\ref{ineqyk}), we get
\begin{align*}
    (1-\varepsilon )\frac{\log x}{\log\log x}<\Diag (x)<(1+\varepsilon )\frac{\log x}{\log\log x}.
\end{align*}
The number $\varepsilon >0$ was arbitrary, so the result follows.
\end{proof}

\section*{Acknowledgements}

The authors wish to thank Carlo Sanna for the idea of the proof of Proposition \ref{P4}.

\bigskip\bigskip

\footnotesize{
\textsc{Institute of Mathematics, Faculty of Mathematics and Computer Science, Jagiellonian University in Krak\'ow, ul. {\L}ojasiewicza 6, 30-348 Krak\'ow, Poland,}

\textit{Email address:} \href{mailto:piotr.miska@uj.edu.pl}{\nolinkurl{piotr.miska@uj.edu.pl}}
}

\vspace{0.5cm}

\footnotesize{
\textsc{Department of Mathematics and Informatics, J. Selye University, P. O. Box 54, 945 01 Kom\'arno, Slovakia,}

\textit{Email address:} \href{mailto:tothj@ujs.sk}{\nolinkurl{tothj@ujs.sk}}
}

\vspace{0.5cm}

\footnotesize{
\textsc{Department of Algebra, Faculty of Mathematics and Physics, Charles University, Sokolovsk\'{a} 83, 18600 Praha 8, Czech Republic,}

\vspace{0.15cm}

\textsc{Institute of Mathematics, Faculty of Mathematics and Computer Science, Jagiellonian University in Krak\'ow, ul. {\L}ojasiewicza 6, 30-348 Krak\'ow, Poland,} 

\textit{Email address:} \href{mailto:blazej.zmija@im.uj.edu.pl}{\nolinkurl{blazej.zmija@im.uj.edu.pl}}
}

\end{document}